\documentclass[12pt]{amsart}

\usepackage{amsmath,amssymb}

\usepackage[colorlinks=true, pdfstartview=FitV, linkcolor=blue, citecolor=blue, urlcolor=blue]{hyperref}
\usepackage{color}
\newtheorem{theorem}{Theorem}[section]
\newtheorem{lemma}[theorem]{Lemma}
\newtheorem{proposition}[theorem]{Proposition}
\newtheorem{corollary}[theorem]{Corollary}

\theoremstyle{definition}

\theoremstyle{remark}

\newtheorem{remark}[theorem]{Remark}
\theoremstyle{example}

\numberwithin{equation}{section}

\newcommand{\thistheoremname}{}
\newtheorem{genericthm}[theorem]{\thistheoremname}

\newcommand{\N}{\mathbb{N}}

\newcommand{\rn}{\mathbb{R}^n}

\begin{document}

\title[Boundedness of operators on weighted Morrey spaces]
{Boundedness of operators on certain power-weighted Morrey spaces beyond the Muckenhoupt weights
}


\author{Javier Duoandikoetxea and Marcel Rosenthal} 
\address{(J. D.) Universidad del Pa\'is Vasco/Euskal Herriko Unibertsitatea, Departamento de Matem\'a\-ti\-cas/Matematika saila, 
Apdo. 644, 48080 Bilbao, Spain}

\email{javier.duoandikoetxea@ehu.eus, marcel.rosenthal@uni-jena.de} 
\subjclass[2010]{Primary }
\thanks{The first author is supported by the grants MTM2014-53850-P of the Ministerio de Econom\'{\i}a y Competitividad (Spain) 
and grant IT-641-13 of the Basque Gouvernment.}


\begin{abstract}
We prove that for operators satistying weighted inequalities with $A_p$ weights the boundedness on a certain class of Morrey spaces holds with weights of the form $|x|^\alpha w(x)$ for $w\in A_p$. In the case of power weights the shift  with respect to the range of Muckenhoupt weights was observed by N.~Samko for the Hilbert transform, by H.~Tanaka for the Hardy-Littlewood maximal operator, and by S.~Nakamura and Y.~Sawano for Calder\'on-Zygmund operators and others. We extend the class of weights and establish the results in a very general setting, with applications to many operators. For weak type Morrey spaces, we obtain new estimates even for the Hardy-Littlewood maximal operator. Moreover, we prove the necessity of certain $A_q$ condition.
\end{abstract}

\maketitle

\section{Introduction}\label{intro}

For $1\le p<\infty$ and $0\le \lambda<n$, let the Morrey space $\mathcal L^{p,\lambda}(w)$ be the collection of all measurable functions $f$ such that 
\begin{equation}\label{normdef}
\|f\|_{\mathcal L^{p,\lambda}(w)}:=\sup_{x\in \rn, r>0}\left(\frac 1{r^\lambda}\int_{B(x,r)}|f|^pw\right)^{1/p}<\infty.
\end{equation}
We also consider the weak Morrey space $W\mathcal L^{p,\lambda}(w)$, for which 
\begin{equation*}
\|f\|_{W\mathcal L^{p,\lambda}(w)}:=\sup_{x\in \rn, r>0, t>0}\left(\frac {t^p\,w(\{y\in B(x,r): |f(y)|>t\})}{r^\lambda}\right)^{1/p}<\infty.
\end{equation*}
(Here and in what follows $w(A)$ stands for the integral of $w$ over $A$.) Clearly, $\mathcal L^{p,\lambda}(w)\subset W\mathcal L^{p,\lambda}(w)$.

N.~Samko proved in \cite{Sam09} that the Hilbert transform is a bounded operator on  $\mathcal L^{p,\lambda}(|x|^{\alpha})$ for $0< \lambda<1$ and $\lambda-1<\alpha<\lambda+p-1$. This range of values of $\alpha$ shows a shift with respect to the corresponding range in the $A_p$ class, which is $-1<\alpha<p-1$. In \cite{Ta15}, H.~Tanaka explored the boundedness on $\mathcal L^{p,\lambda}(w)$ of the Hardy-Littlewood maximal operator and was able to describe necessary conditions and sufficient conditions, but not a characterization. Nevertheless, for power weights $w(x)=|x|^{\alpha}$ he obtained the sharp range $\lambda-n\le\alpha<\lambda+n(p-1)$, which in the one-dimensional case coincides with the range obtained by Samko for the Hilbert transform except at the endpoint $\alpha=\lambda-n$. Later on, S.~Nakamura and Y.~Sawano in \cite{NkSa17} studied the boundedness of the Riesz transforms and other singular integrals and obtained similar shifted ranges for the case of $\mathcal L^{p,\lambda}(|x|^{\alpha})$ (with open left endpoint). 

In \cite{DuRo} the authors of this paper proved a general result involving Muckenhoupt weights, under the assumptions of the extrapolation theorem for $A_p$ weights. When particularized for the Hardy-Littlewood maximal operator or for Calder\'on-Zygmund operators, the boundedness on $\mathcal L^{p,\lambda}(w)$ was obtained for $w\in A_p\cap RH_\sigma$ in the range $0\le \lambda\le n/\sigma'$, which for $w(x)=|x|^{\alpha}$ gives the range  $\lambda-n\le\alpha<n(p-1)$. In this paper the results in \cite{DuRo} for $\mathcal L^{p,\lambda}(w)$ spaces are extended to the weights $|x|^\alpha w(x)$ for $w\in A_p\cap RH_\sigma$ and $0<\alpha<\lambda$. (We recall the definitions of the weight classes $A_p$ and the reverse H\"older classes $RH_\sigma$ in Section \ref{sec2}.) In particular, we prove the following theorem.

\begin{theorem}\label{teo1}
Let $1\le p_0<\infty$ and let $\mathcal F$ be a collection of nonnegative measurable pairs of functions. Assume that for every $(f,g)\in \mathcal F$ and every $w\in A_{p_0}$ we have
\begin{equation}\label{hyp1}
\|g\|_{L^{p_0}(w)}\le C \|f\|_{L^{p_0}(w)},
\end{equation}
where $C$ does not depend on the pair $(f,g)$ and it depends on $w$ only in terms of $[w]_{A_{p_0}}$ (defined at the beginning of Section \ref{sec2}). Then for $1<p<\infty$ and $w\in A_p\cap RH_\sigma$ it holds
 \begin{equation}\label{bound1}
\|g\|_{\mathcal L^{p,\lambda}(|x|^\alpha w)}\le C \|f\|_{\mathcal L^{p,\lambda}(|x|^\alpha w)},
\end{equation}
for $0\le \lambda<n/\sigma'$ and $0\le \alpha<\lambda$. In particular, for power weights of the form $|x|^\beta=|x|^\alpha w(x)$, the estimate  \eqref{bound1} holds 
for $\lambda-n<\beta<\lambda+n(p-1)$. If the hypothesis holds for $p_0=1$, then \eqref{bound1} also holds for $p=1$, and  in the case of power weights for $\lambda-n<\beta<\lambda$.
\end{theorem}

When we say that \eqref{hyp1} holds for every $w\in A_{p_0}$ we mean that if the right-hand side is finite for a fixed $w$, then also the left-hand side is finite for the same $w$ and the inequality holds. The conclusion of the theorem is to be understood in the same way: if $f$ is in $\mathcal L^{p,\lambda}(|x|^\alpha w)$, then $g$ is in the same space and the inequality holds.

To make things clear let us say that the weights appearing in the theorem are always in some Muckenhoupt class. Indeed, $|x|^\alpha w\in A_{p+ \lambda/ n}$ for $0\le \alpha<\lambda$ and $w\in A_p$.  Being in $A_{p+ \lambda/ n}$ is not a particular restriction for the weights in our theorem, because we show in Proposition \ref{necrange} the necessity  of $u\in A_{p+ \lambda/ n}$ for the boundedness of the Hardy-Littlewood maximal operator in the Morrey space $\mathcal L^{p,\lambda}(u)$ (and even for the weak-type boundedness). We are thus forced to deal with Muckenhoupt weights. When we say that the results go beyond the Muckenhoupt range, we mean that for a fixed value of $p$ the boundedness of the involved operators holds for weights which are not necessarily in $A_p$. 

This theorem has a number of applications because we know that for many operators $T$, the pairs $(|f|, |Tf|)$ satisfy its assumptions. In particular, we recover the results for the Hardy-Littlewood maximal operator (except the left endpoint, that is, $\beta=\lambda-n$), the Hilbert transform and the Calder\'on-Zygmund operators mentioned above. But it extends also to Littlewood-Paley operators, rough singular integrals and others. Moreover, in those examples the case $p=1$ of the theorem provides a weak-type result, from $\mathcal L^{1,\lambda}(|x|^\beta)$ to $W\mathcal L^{1,\lambda}(|x|^\beta)$ for $\lambda-n<\beta<\lambda$, which for $\beta>0$ is new even for the Hardy-Littewood maximal operator.

We present some preliminary results in Section \ref{sec2}. The proof of Theorem \ref{teo1} is in Section \ref{sec3}, where we also prove another theorem suited to operators satisfying the assumptions of the so-called limited range extrapolation. In Section \ref{sec4new} we establish embeddings which allow to define the operators in the Morrey spaces by restriction. In Section \ref{sec4} we prove the necessity of the $A_{p+ \lambda/ n}$ condition for the Hardy-Littlewood maximal operator and of $A_{p+ \lambda}$ for the Hilbert transform in the case of weak-type estimates. This implies the necessity of the range of power weights for positive exponents. We also give an easy proof of the necessity for negative exponents. In the case of the strong estimates and power weights this was proved  by Tanaka in \cite{Ta15} checking his more general necessary condition. We extend the necessity to the weak-type estimates. In the same section we prove the estimate for the left endpoint for power weights (that is, for the weight $|x|^{\lambda-n}$).

\section{Preliminary results}\label{sec2}

Let 
$w\in L_1^\textrm{loc}(\rn)$ with $w>0$ almost everywhere. We say that $w$ is a \textit{Muckenhoupt weight} belonging to $A_p$ for $1<p<\infty$ if
\begin{equation*} 
  [w]_{A_p}\equiv 
	\sup_B \frac{w(B)}{|B|} \left( \frac{w^{1-p'}(B)}{|B|} \right)^{p-1}<\infty,
\end{equation*}
where the supremum is taken over all Euclidean balls $B$ in $\rn$. The quantity  $[w]_{A_p}$ is the $A_p$ constant of $w$.
We say that $w$  belongs to $A_1$ if, for any Euclidean ball $B$,
\begin{equation*} 
  \frac{w(B)}{|B|}\le c w(x) \text{ for almost all } x\in B.
\end{equation*}
The $A_1$ constant of $w$, denoted by  $[w]_{A_1}$, is the smallest constant $c$ for which the inequality holds. 

We say that a nonnegative locally integrable function $w$ on $\rn$ belongs to the \textit{reverse H\"older class} $RH_\sigma$ for $1<\sigma<\infty$ if it satisfies the \textit{reverse H\"older inequality} with exponent $\sigma$, i.e., \begin{equation*}
  \left(\frac{1}{|B|}\int_{B} w(x)^\sigma dx \right)^\frac{1}{\sigma}\le \frac{c}{|B|} \int_{B} w(x) dx,
\end{equation*}
where the constant $c$ is universal for all Euclidean balls $B\subset \rn$.

\begin{remark} \label{propap} Some results for weights  are the following. We use repeatedly the first two throughout the paper.
\begin{itemize}
\item[(i)] If $Mh<\infty$ a.e., then $(Mh)^{1/s}\in A_1$ and its $A_1$ constant depends on $s$, but not on $h$. Moreover, $(Mh)^{1/s}\in A_1\cap RH_\sigma$ if $s>\sigma$.
\item[(ii)] Let $w\in RH_\sigma$. For any ball $B$ and any measurable $E\subset B$ it holds that 
\begin{equation}\label{ineqRH}
 \frac{w(E)}{w(B)}\le c \left(\frac{|E|}{|B|}\right)^{1/\sigma'}. 
\end{equation}
 Since $w\in A_p$ implies that $w\in RH_\sigma$ for some $\sigma$, the inequality holds for each $A_p$ weight for the appropriate $\sigma$. 
\item[(iii)] Weights simultaneously in $A_p$ and $RH_\sigma$ can be described  (\cite{JN91}) as
\begin{equation}\label{apandrh}
A_p\cap RH_\sigma =\{w: w^\sigma\in A_{\sigma(p-1)+1}\}.
\end{equation}
\end{itemize}

\end{remark}

\begin{remark}\label{select}
When dealing with the definition of the norm \eqref{normdef}, we only need to take into account two types of balls: balls centered at the origin and balls of the form $B(x,r)$ with $r<|x|/4$. Indeed, if we have a ball $B(x,r)$ with $r\ge |x|/4$, it holds that $B(x,r)\subset B(0,5r)$, and since the radii are comparable we can replace the smaller ball by the larger one.
\end{remark}

The following lemma provides an estimate which is used in the proofs of the theorems.
\begin{lemma}
 Let $1\le p<\infty$ and $0\le \lambda <n$. Let $f\ge 0$ in $\mathcal L^{p,\lambda}(|x|^\alpha w)$ and $r>0$. If $\alpha<\lambda$, then
 \begin{equation}\label{basicineq}
\left(\int_{B(0,r)}f^{p} w\right)^{\frac 1{p}}\le C r^{\frac{\lambda-\alpha}p} \|f\|_{\mathcal L^{p,\lambda}(|x|^\alpha w)}.
\end{equation}
The constant depends only on $\alpha, \lambda$ and  $p$.
\end{lemma}

\begin{proof}
Let $A_j=B(0, 2^{-j+1}r)\setminus  B(0, 2^{-j}r)$, $j\in \N$. Then
\begin{align*}
\int_{B(0,r)}f^p w&\le C\sum_{j=1}^\infty \int_{A_j}f(y)^p w(y)\left(\frac{|y|}{2^{-j}r}\right)^\alpha dy
\\
&\le C\sum_{j=1}^\infty (2^{-j}r)^{\lambda-\alpha} \|f\|^p_{\mathcal L^{p,\lambda}(|x|^\alpha w)}.
\end{align*}
If $\alpha<\lambda$ the series is convergent and \eqref{basicineq} follows.
\end{proof}

\begin{remark}
 We deal with integrals of the type $\int_{B} |f|^pw$ where $w$ is a certain $A_p$ weight and $B$ is a ball. This can be written as $\int |f|^pw\chi_{B}$. But $w\chi_{B}$ cannot be an $A_p$ weight for any $p$ because $A_p$-weights cannot vanish in a set of positive measure. Hence, the second proof of Theorem 6.1 of \cite{Ad15} is not correct, because it claims that $r^{\lambda-n}\chi_{B(x,r)}$ is an $A_1$ weight.
\end{remark}

\section{Main theorems}\label{sec3}

In this section we first prove Theorem \ref{teo1}, starting with the assumption for $p_0=1$. This case is important because the proof of the theorem is simpler and the general case $p_0\in (1,\infty)$ can be reduced to this one by a scaling argument. In the applications it is also significant because for a number of operators weighted weak-type $(1,1)$ estimates are known and our result provides weak-type Morrey estimates for them.

\begin{proof}[Proof of Theorem \ref{teo1}] 
\textit{Case $p_0=1$}. We assume first that \eqref{hyp1} holds with $p_0=1$. 

Let $B_r:=B(x,r)$ be one of the balls considered in Remark \ref{select}. Let $w\in A_p$. We have
\begin{equation}\label{uno}
\int_{B_r}g(y)^p|y|^\alpha w(y) dy \le 
\begin{cases} 
r^\alpha \displaystyle\int_{B_r}g^p w, &\text{if $x=0$;}\\[10pt]
C|x|^\alpha \displaystyle\int_{B_r}g^p w, &\text{if $0<r\le |x|/4$.}
\end{cases} 
\end{equation}

\textit{Proof for $1<p<\infty$}.
In both cases we are left with the integral of $g^pw$ on the ball $B_r$, which we handle as in \cite{DuRo}. Using duality we have
\begin{equation*}
\left(\int_{B_r} g^{p}w\right)^\frac{1}{p} =\sup_{h\,:\,\| h\|_{L^{p'}(w,B_r)}=1}\int_{B_r} gh w.
	\end{equation*}
Fix such a function $h$ and we have for $s>1$ that
\begin{equation} \label{p2}
\int_{B_r} gh w  \le  \int_{\rn} gM(h^s w^s \chi_{B_r})^\frac{1}{s}
 \le c \int_{\rn} fM(h^s w^s \chi_{B_r})^\frac{1}{s},  
\end{equation}
provided that $M(h^s w^s \chi_{B_r})^\frac{1}{s}\in A_1$. According to Remark \ref{propap} (i) we need $M(h^s w^s \chi_{B_r})(x)<\infty$ a.e. We check that $h^s w^s \chi_{B_r}\in L^1$ for appropriate $s>1$, and get a bound for its integral for future use. To this end, we choose $s>1$ such that $w^{1-p'}\in A_{p'/s}$, which is possible because $w^{1-p'}\in A_{p'}$. We have
\begin{align} 
\begin{split} \label{p3}
	 \left(\int_{B_r} h^s w^{s-1} w\right)^\frac{1}{s}
	&\le\left(\int_{B_r} h^{p'} w\right)^\frac{1}{p'}
	\left(\int_{B_r} w^{(s-1) \frac{p'}{p'-s} +1} \right)^{\frac1{s}-\frac{1}{p'}}\\ 
	& \le c\  |B_r|^{\frac{1}{s}}w^{1- p'}(B_r)^{-\frac{1}{p'}}
		\le c\ w(B_r)^{\frac{1}{p}}r^{-\frac{n}{s'}},
\end{split} 
\end{align}
where the second inequality holds because $w^{1-p'}\in A_{p'/s}$ (the exponent of $w$ in the integral is the same as $(1-p')(1-(p'/s)')$) and in the last one we use 
\begin{equation*}
c_nr^n=|B_r|\le w(B_r)^{\frac{1}{p}} w^{1-p'}(B_r)^{\frac{1}{p'}}.
\end{equation*}

We split the last integral of \eqref{p2} into the integral over $B_{2r}$ and over its complement.
On the one side we have 
\begin{equation*} 
\int_{B_{2r}} f  M(h^s w^s \chi_{B_r})^\frac{1}{s} \le 
	\left(\int_{B_{2r}} f^{p}w\right)^{\frac1{p}}\left(\int_{B_{2r}} M(h^s w^s \chi_{B_{r}})^{\frac{p'}s}w^{1- p'}\right)^{\frac 1{p'}}.
\end{equation*}
The last term is bounded by a constant because  $M$ is bounded on $L^{p'/s}(w^{1-p'})$ and we get a constant times the norm of $h$ in $L^{p'}(w, B_r)$, which is $1$. 
Now we have
\begin{equation}\label{unobis}
\int_{B_{2r}} f^{p}w\le \begin{cases} 
C (2r)^{\lambda-\alpha} \|f\|^p_{\mathcal L^{p,\lambda}(|x|^\alpha w)}, &\text{if $x=0$;}\\[10pt]
C|x|^{-\alpha}(2r)^{\lambda} \|f\|^p_{\mathcal L^{p,\lambda}(|x|^\alpha w)}, &\text{if $0<r\le |x|/4$;}
\end{cases} 
\end{equation}
where in the first case we use \eqref{basicineq}, and in the second case we use $|x|\sim |y|$ for $y\in B_{2r}$.

To deal with the integral on $\rn\setminus B_{2r}$ we decompose it into annuli and use that on $B_{2^{j+1}r}\setminus B_{2^{j}r}$ the maximal operator is  comparable to $(2^j r)^{-n}\int_{B_r} h^sw^s$. Using \eqref{p3} we get 
\begin{equation}\label{dos}
\aligned
\sum_{j=1}^\infty &\int_{B_{2^{j+1}r}\setminus B_{2^{j}r}} f 
\  \frac{w(B_r)^{\frac 1 {p}}}{2^{jn/s}r^n}\\
& \le
\sum_{j=1}^\infty \left(\int_{B_{2^{j+1}r}\setminus B_{2^{j}r}} f^{p} w \right)^\frac{1}{p}  \frac{w^{1- p'} \left(B_{2^{j+1}r}\right)^{\frac 1{p'}}w(B_r)^{\frac 1{p}}}{2^{jn/s}r^n}\\
& \le\sum_{j=1}^\infty \left(\int_{B_{2^{j+1}r}} f^{p} w \right)^\frac{1}{p} 2^{\frac{jn}{s'}} 2^{-\frac{jn}{p\sigma'}}, 
\endaligned
\end{equation}
where in the last step we use
\begin{equation*}
w\left(B_{2^{j+1}r} \right)^{\frac 1{p}} w^{1-p'} \left(B_{2^{j+1}r}\right)^{\frac 1{ p'}}\le C |B_{2^{j+1}r}|=C'2^{jn}r^n,
\end{equation*}
and \eqref{ineqRH} for $B_r$ and $B_{2^{j+1}r}$.

If $B_r$ is centered at the origin, we use \eqref{basicineq} in the last term of \eqref{dos} to get
\begin{equation}\label{tres}
C\sum_{j=1}^\infty 2^{j[(\lambda-\alpha-\frac{n}{\sigma'})\frac{1}{p}+\frac{n}{s'}]}r^{(\lambda-\alpha)\frac{1}{p}}\|f\|_{\mathcal L^{p,\lambda}(|x|^\alpha w)}.
\end{equation}
The series is convergent if we choose $s$ close enough to $1$, because we assume $\lambda<{n}/{\sigma'}$.

If the ball is centered at $x$ and $r\le |x|/4$, we can assume $|x|=2^{N}r$ for some $N\ge 2$, increasing slightly $r$ if necessary. For the integral in \eqref{dos} we distinguish the cases $j\le N-2$ and $j\ge N-1$. In the first case, if $y\in B_{2^{j+1}r}$, then $|y|\sim |x|$. In the second case, $B_{2^{j+1}r}\subset B(0, 2^{j+2}r)$. As a consequence, for $j\le N-2$, 
\begin{equation*}
\aligned
\int_{B_{2^{j+1}r}} f^{p} w&\le C |x|^{-\alpha} \int_{B_{2^{j+1}r}} f(y)^{p} w(y) |y|^\alpha dy\\
&\le C |x|^{-\alpha}(2^jr)^\lambda \|f\|^p_{\mathcal L^{p,\lambda}(|x|^\alpha w)}.
\endaligned
\end{equation*}
For $j\ge N-1$ we use \eqref{basicineq} and we obtain
\begin{equation*}
\int_{B_{2^{j+1}r}} f^{p} w\le \int_{B(0, 2^{j+2}r)} f^{p} w 
\le C (2^{j}r)^{\lambda-\alpha} \|f\|^p_{\mathcal L^{p,\lambda}(|x|^\alpha w)}.
\end{equation*}
Inserting this into \eqref{dos} we obtain a constant times
\begin{equation*}
\|f\|_{\mathcal L^{p,\lambda}(|x|^\alpha w)}\left(\sum_{j=1}^{N-2} |x|^{-\frac{\alpha}p} r^{\frac{\lambda}p}2^{j[(\lambda-\frac{n}{\sigma'})\frac{1}{p}+\frac{n}{s'}]}+\sum_{N-1}^{\infty}r^{(\lambda-\alpha)\frac 1p}2^{j[(\lambda-\alpha-\frac{n}{\sigma'})\frac{1}{p}+\frac{n}{s'}]}\right).
\end{equation*}
The first sum is bounded independently of $N$ if we take $s$ close enough to $1$, because $\lambda<{n}/{\sigma'}$. For this same reason the series is convergent and the value of its sum is
\begin{equation*}
C2^{N[(\lambda-\alpha-\frac{n}{\sigma'})\frac{1}{p}+\frac{n}{s'}]} r^{(\lambda-\alpha)\frac 1p} =C|x|^{-\frac{\alpha}p} r^{\frac{\lambda}p} 2^{N[(\lambda-\frac{n}{\sigma'})\frac 1p+\frac{n}{s'}]}\le C'|x|^{-\frac{\alpha}p} r^{\frac{\lambda}p},
\end{equation*}
where $C'$ is independent of $N$ because the exponent of $2^N$ in the middle term is negative. Taking into account \eqref{unobis}, \eqref{tres} and the recent bound, the right-hand side of \eqref{uno} is bounded as desired.

\textit{Proof for  $p=1$}. The proof is similar but easier because we do not need to use a duality argument and there is no $h$ as in the previous situation. 

We have $w\in A_1\cap RH_\sigma$ (which is the same as $w^\sigma\in A_1$). In the construction of the $A_1$ weight $M(w^s\chi_{B_r})^{1/s}$ we choose $1<s<\sigma$. Then $M(w^s\chi_{B_r})^{1/s}\le M(w^s)^{1/s}\le w$ a.e. and when we integrate on $B_{2r}$ we obtain \eqref{unobis} with $p=1$.

When we are in $B_{2^{j+1}r}\setminus B_{2^{j}r}$ we have 
\begin{equation*}
M(w^s\chi_{B_r})^{1/s}\le C \left(\frac {w^s(B_r)}{(2^jr)^n}\right)^{1/s}.
\end{equation*}
Using that $w\in RH_s$ (because $s<\sigma$), $w^s(B_r)^{1/s}\le C |B_r|^{-1/s'}w(B_r)$. 
Instead of \eqref{dos} we have now
\begin{equation*}
\sum_{j=1}^\infty \int_{B_{2^{j+1}r}\setminus B_{2^{j}r}} f w\  2^{jn/s'} \frac{w(B_r)}{w(B_{2^{j+1}r})}\le \sum_{j=1}^\infty \int_{B_{2^{j+1}r}\setminus B_{2^{j}r}} f w\  2^{j(\frac n{s'}-\frac n{\sigma'})}.
\end{equation*}
The proof continues as before.

\textit{Case $p_0>1$}. By the usual extrapolation theorem the assumption \eqref{hyp1} is valid for any $p_0\in (1,\infty)$. Given $p>1$ and $w\in A_p\cap RH_{\sigma}$, we choose $p_0>1$ for which $w\in A_{p/p_0}$. The assumption holds in the form
\begin{equation*}
\|g^{p_0}\|_{L^{1}(v)}\le C \|f^{p_0}\|_{L^{1}(v)},
\end{equation*}
for $v\in A_1\subset A_{p_0}$ and we can apply the previous part of the proof to the pair $(f^{p_0},g^{p_0})$ to get the Morrey estimate with exponent $p/p_0$.  
\end{proof}

We can generalize Theorem \ref{teo1} to a setting in which weighted inequalities is a restricted range are assumed. 

\begin{theorem}\label{teo2}
Let $1<b< \infty$ and $1\le p_0<b$. Let $\mathcal F$ be a collection of nonnegative measurable pairs of functions. Assume that for every $(f,g)\in \mathcal F$ and every $w\in A_{p_0}\cap RH_{(b/p_0)'}$ we have
\begin{equation}\label{hyp2}
\|g\|_{L^{p_0}(w)}\le C \|f\|_{L^{p_0}(w)},
\end{equation}
where $C$ does not depend on the pair $(f,g)$ and it depends on $w$ only in terms of the $A_{p_0}$ and $RH_{(b/p_0)'}$ constants of $w$. Then if $1<p<b$ and $w\in A_{p}\cap RH_{\sigma}$ with $\sigma> (b/p)'$, it holds that 
 \begin{equation*} 
\|g\|_{\mathcal L^{p,\lambda}(|x|^\alpha w)}\le C \|f\|_{\mathcal L^{p,\lambda}(|x|^\alpha w)},
\end{equation*}
for  $0\le \lambda<n(\frac 1{\sigma'}-\frac pb)$ and $0\le \alpha<\lambda$. In particular, for power weights of the form $|x|^\beta=|x|^\alpha w(x)$ and for 
\begin{equation*}
\lambda-n\left(1-\frac pb\right)<\beta<\lambda+n(p-1),
\end{equation*}
it holds that
 \begin{equation*}
\|g\|_{\mathcal L^{p,\lambda}(|x|^\beta)}\le C \|f\|_{\mathcal L^{p,\lambda}(|x|^\beta)}.
\end{equation*}
Moreover, if the hypothesis holds for $p_0=1$, then the results are valid for $p=1$.
\end{theorem}

\begin{proof}
 The proof is similar to that of Theorem \ref{teo1}. Starting with $p_0=1$, to apply the hypothesis in \eqref{p2} we need to assume $s>b'$ because the weight has to be in $A_1\cap RH_{b'}$. The estimate \eqref{p3} and the boundedness of $M$ in $L^{p'/s}(w^{1-p'})$ need $w^{1-p'}\in A_{p'/s}$. We are assuming that $w\in A_{p}\cap RH_{\sigma}$ with $\sigma> (b/p)'$, in particular, $w\in A_{p}\cap RH_{(b/p)'}$. According to \eqref{apandrh} this is the same as saying $w^{(b/p)'}\in A_{(b/p)'(p-1)+1}$, which by duality yields $w^{(b/p)'(1-q')}\in A_{q'}$ with $q=(b/p)'(p-1)+1$, that is, $w^{1-p'}\in A_{p'/b'}$. Then there exists $s>b'$ for which $w^{1-p'}\in A_{p'/s}$ as needed. The proof continues as before, and we only need to add the condition that makes the series convergent. This condition is   $\lambda<n(\frac 1{\sigma'}-\frac pb)$. 
 
If we assume $p_0>1$ in \eqref{hyp2}, by the usual extrapolation theorem we can consider any $p_0\in (1,b)$. Given $p$ and $w$ we proceed again as before by choosing $p_0$ close enough to $1$ such that $w\in A_{p/p_0}$ and working with the pairs $(f^{p_0}, g^{p_0})$. 
\end{proof}

The formulation of the extrapolation theorem in terms of pairs of functions provides several extensions as corollaries (see \cite[p. 21--22]{CMP11}). In a similar way, we can get similar extensions in the Morrey setting. We state the scaling and weak-type extensions in the following two corollaries, and leave to the interested reader the extension to the  vector-valued setting.

\begin{corollary}\label{cor3}
Let $0<p_-\le p_0<p_+\le \infty$. Let $\mathcal F$ be a collection of nonnegative measurable pairs of functions. Assume that for every $(f,g)\in \mathcal F$ and every $w\in A_{\frac{p_0}{p_-}}\cap RH_{\left(\frac{p_+}{p_0}\right)'}$ we have
\begin{equation}\label{hyp3}
\|g\|_{L^{p_0}(w)}\le C \|f\|_{L^{p_0}(w)},
\end{equation}
where $C$ does not depend on the pair $(f,g)$ and it depends on $w$ only in terms of the $A_{\frac{p_0}{p_-}}$ and $RH_{\left(\frac{p_+}{p_0}\right)'}$ constants of $w$. Then if $p_-<p<p_+$ and $w\in A_{\frac{p}{p_-}}\cap RH_{\sigma}$ with $\sigma> {\left(\frac{p_+}{p}\right)'}$, it holds that 
 \begin{equation*} 
\|g\|_{\mathcal L^{p,\lambda}(|x|^\alpha w)}\le C \|f\|_{\mathcal L^{p,\lambda}(|x|^\alpha w)},
\end{equation*}
for  $0\le \lambda<n(\frac 1{\sigma'}-\frac {p}{p_+})$ and $0\le \alpha<\lambda$. In particular, for power weights of the form $|x|^\beta=|x|^\alpha w(x)$ and for 
\begin{equation*}
\lambda-n\left(1-\frac {p}{p_+}\right)<\beta<\lambda+n(\frac {p}{p_-}-1),
\end{equation*}
it holds that
 \begin{equation*}
\|g\|_{\mathcal L^{p,\lambda}(|x|^\beta)}\le C \|f\|_{\mathcal L^{p,\lambda}(|x|^\beta)}.
\end{equation*}
Moreover, if the hypothesis holds for $p_0=p_-$, then the conclusion is valid for $p=p_-$.
\end{corollary}

If $p_+=\infty$ this is a corollary to Theorem \ref{teo1} and if $p_+<\infty$ to Theorem \ref{teo2}. The proof is immediate from the theorems if we read \eqref{hyp3} as
\begin{equation*}
\|g^{p_-}\|_{L^{\widetilde{p_0}}(w)}\le C \|f^{p_-}\|_{L^{\widetilde{p_0}}(w)},
\end{equation*}
for every $w\in A_{\widetilde{p_0}}\cap RH_{(b/\widetilde{p_0})'}$ with $\widetilde{p_0}=\frac{p_0}{p_-}$ and $b=\frac{p_+}{p_-}$. 

\begin{corollary}
If in Theorem \ref{teo1}, Theorem \ref{teo2} or Corollary \ref{cor3} the assumptions hold as weak-type inequalities, that is,  with $\|g\|_{L^{p_0,\infty}(w)}$ instead of \break $\|g\|_{L^{p_0}(w)}$, then the conclusions also hold in the weak sense, that is, with  $\|g\|_{W\mathcal L^{p,\lambda}(|x|^\alpha w)}$ instead of $\|g\|_{\mathcal L^{p,\lambda}(|x|^\alpha w)}$.
\end{corollary}

To prove this case, the weak-type hypothesis can be read as a strong type inequality for the pair $(f, t\chi_{\{g>t\}})$ with constants uniform in $t$. 

\section{Embeddings and applications}\label{sec4new}

The proof of Theorem \ref{teo1} shows that 
\begin{equation*}
\int_{\rn} fM(h^s w^s \chi_{B_r})^\frac{1}{s}\le C \|f\|_{\mathcal L^{p,\lambda}(|x|^\alpha w)},
\end{equation*}
for $0\le \lambda<n$, $w\in A_p$ and appropriate $s>1$. This implies the continuous embedding 
\begin{equation*}
\mathcal L^{p,\lambda}(|x|^\alpha w)\hookrightarrow L^1(M(h^s w^s \chi_{B_r})^\frac{1}{s}). 
\end{equation*}
In particular, choosing the ball $B(0,1)$ and $h=cw^{-1}$, we have  \break $\mathcal L^{p,\lambda}(|x|^\alpha w)\hookrightarrow L^1((1+|x|)^{-\beta})$ for some $\beta<n$. Since $M(h^s w^s \chi_{B_r})^\frac{1}{s}\in A_1$, we have 
\begin{equation*}
\mathcal L^{p,\lambda}(|x|^\alpha w)\subset \bigcup_{u\in A_1} L^1(u).
\end{equation*}
By the scaling argument at the end of the proof of the same theorem, if $p>1$, 
\begin{equation}\label{emb2}
\mathcal L^{p,\lambda}(|x|^\alpha w)\subset \bigcup_{u\in A_q} L^q(u), \quad q>1.
\end{equation}
(The right-hand side is independent of $q$; see \cite{KMM16} and \cite{DuRo}.) Suitable embeddings can be written for the weighted spaces appearing in Theorem \ref{teo2}, hence in Corollary \ref{cor3}.  

The applications to boundedness of operators in Morrey spaces are corollaries of the general theorems and of the embeddings just mentioned. For instance, the basic one is the following.

\begin{corollary}\label{korol}
Assume that for some $p_0\in [1,\infty)$, $T$ is an operator acting from $\bigcup_{w\in A_{p_0}} L^{p_0}(w)$ into the space of measurable functions and satisfying
\begin{equation} \label{ex1}
\|Tf\|_{L^{p_0}(w)}\le C \|f\|_{L^{p_0}(w)} 
\end{equation}
 for all $f\in L^{p_0}(w)$ and $w\in A_{p_0}$, with a constant depending on $[w]_{A_{p_0}}$.
Then for every
$1< p<\infty$ (and also $p=1$ if $p_0=1$), $w\in A_{p}\cap RH_\sigma$, $0\le \lambda<n/\sigma'$ and $0<\alpha<\lambda$,  we have
that $T$ is well defined on $\mathcal L^{p,\lambda}(|x|^\alpha w)$ by restriction and, moreover,
 \begin{equation}\label{bound13}
\|Tf\|_{\mathcal L^{p,\lambda}(|x|^\alpha w)}\le C \|f\|_{\mathcal L^{p,\lambda}(|x|^\alpha w)}.
\end{equation}
For power weights $|x|^\beta$, $T$ is well defined and bounded on $\mathcal L^{p,\lambda}(|x|^\beta)$ if $\lambda-n<\beta<\lambda+n(p-1)$.

If \eqref{ex1} is replaced by the weak estimate from $L^1(w)$ to $L^{1,\infty}(w)$ for $w\in A_1$, then \eqref{bound13} holds from $\mathcal L^{1,\lambda}(|x|^\alpha w)$ to $W\mathcal L^{1,\lambda}(|x|^\alpha w)$.
\end{corollary}

The definition by embedding is guaranteed by \eqref{emb2} and the size estimate by Theorem \ref{teo1}. 

There are many operators satisfying the assumptions of the theorem: the Hardy-Littlewood maximal operator, Calder\'on-Zygmund operators, rough operators with kernel $|x|^{-n}\Omega(x/|x|)$ with $\Omega\in L^\infty(S^{n-1})$ and integral zero, commutators (in this case the weighted weak-type $(1,1)$ does not hold), square functions (including some Littlewood-Paley type operators, Lusin area integral, $g_\lambda$ functions, Marcinkiewicz integral), etc. 

Similar corollaries can be written for our other general results.  All the applications mentioned in \cite{DuRo} for the spaces $\mathcal L^{p,\lambda}(w)$ are now extended to $\mathcal L^{p,\lambda}(|x|^\alpha w)$ with $0\le\alpha<\lambda$ by the theorems in this paper.  Note that in \cite{DuRo} the space $\mathcal L^{p,\nu}(w)$ was denoted as $L_p^r(\lambda,w,\mathbb R^n)$ with $n+rp=\nu$.

\section{Necessary conditions for $M$ and $H$, and the endpoint for $M$}\label{sec4}

Tanaka proved in \cite{Ta15} that $M$ is bounded on $L^p(|x|^\beta)$ if and only if $\lambda-n\le \beta<\lambda+n(p-1)$. The necessity of the upper bound means that $0\le \alpha<\lambda$ in Theorem \ref{teo1} is optimal. Tanaka's proof uses a general necessary condition involving duality in Morrey spaces. Avoiding duality, we prove first a necessary condition in terms of the $A_q$ scale from which the necessity of the upper bound follows, and next the necessity of the lower bound is proved in a direct way. In all cases our necessary conditions are valid for the weak-type $(p,p)$ estimates. The sufficiency for $\lambda-n< \beta<\lambda+n(p-1)$ comes from Theorem \ref{teo1}. We give in Proposition \ref{endp} below a direct proof of the boundedness at the endpoint $\mathcal L^{p,\lambda}(|x|^{\lambda-n})$ for $1<p<\infty$ and the corresponding weak estimate for $p=1$. The weak estimates are not in \cite{Ta15}.

\begin{lemma}\label{emb5}
Let $1\le p<\infty$ and $0\le \lambda <n$.  The embedding $L^{\frac {pn}{n-\lambda}}(w^{\frac{n}{n-\lambda}})\hookrightarrow \mathcal L^{p,\lambda}(w)$ holds with constant depending only on $n$, $\lambda$ and $p$, not on $w$.
\end{lemma}

\begin{proof}
Let $B$ be a ball of radius $r$. Then
\begin{equation*}
\frac 1{r^\lambda}\int_{B}|f|^pw\le \frac 1{r^\lambda}\left(\int_{B}|f|^{\frac{pn}{n-\lambda}}w^{\frac{n}{n-\lambda}}\right)^{1-\frac \lambda n} |B|^{\frac \lambda n}\le c_{n,\lambda} \|f\|_{L^{\frac {pn}{n-\lambda}}(w^{\frac{n}{n-\lambda}})}^p. \qedhere
\end{equation*} 
\end{proof}

\begin{proposition}\label{necrange}
Let $1\le p<\infty$ and $0\le \lambda <n$. If $M$ is bounded from $\mathcal L^{p,\lambda}(w)$ to $W\mathcal L^{p,\lambda}(w)$, then $w\in A_{p+\lambda/n}$. 
\end{proposition}

\begin{proof}
Let $B$ be a ball of radius $r$. Define $f=\sigma \chi_B$ with $\sigma$ nonnegative to be chosen later. For $x\in B$, we have $Mf(x)\ge \sigma(B)/|B|$. If $t<\sigma(B)/|B|$, then $B= \{x\in B: Mf(x)>t\}$. Assuming that $M$ is bounded from $\mathcal L^{p,\lambda}(w)$ to $W\mathcal L^{p,\lambda}(w)$ we have 
\begin{equation*}
\aligned
\frac{t w(B)^{1/p}}{r^{\lambda/p}}\le C \|\sigma \chi_B\|_{\mathcal L^{p,\lambda}(w)}&\le C'  \|\sigma \chi_B\|_{L^{\frac {pn}{n-\lambda}}(w)}\\
&=C'\left(\int_B \sigma^{\frac{pn}{n-\lambda}}w^{\frac{n}{n-\lambda}}\right)^{\frac {n-\lambda}{pn}},
\endaligned
\end{equation*}
where we used Lemma \ref{emb5} in the second  inequality. 
Let $t$ tend to $\sigma(B)/|B|$ and choose $\sigma$ such that $\sigma=\sigma^{\frac{pn}{n-\lambda}}w^{\frac{n}{n-\lambda}}$, that is, $\sigma^{1-p-\frac \lambda n}=w$. We get 
\begin{equation*}
\frac{w(B) \sigma(B)^{p+\frac \lambda n-1}}{|B|^{p+\frac \lambda n}}\le C,
\end{equation*}
with a constant independent of $B$. Therefore, $w\in A_{p+\lambda/n}$. 

To be precise, we do not know a priori that $\sigma(B)$ is finite for the choice $\sigma^{1-p-\frac \lambda n}=w$. As usual, to overcome this problem, we define $\sigma_\epsilon$ by  $\sigma_\epsilon^{1-p-\frac \lambda n}=w+\epsilon$ for $\epsilon>0$ and let $\epsilon$ tend to $0$.  
\end{proof}

\begin{proposition}\label{necrangepower}
Let $1\le p<\infty$ and $0\le \lambda <n$. If $M$ is bounded from $\mathcal L^{p,\lambda}(|x|^\beta)$ to $W\mathcal L^{p,\lambda}(|x|^\beta)$, then $\lambda-n\le \beta<\lambda+n(p-1)$. 
\end{proposition}

\begin{proof}
According to the previous proposition, $|x|^\beta\in A_{p+\lambda/n}$ and this requires $\beta<\lambda+n(p-1)$. 

For the lower bound, first we observe that if the characteristic function of a ball centered at the origin is in $W\mathcal L^{p,\lambda}(|x|^\beta)$, then $\lambda-n\le \beta$. Indeed, let $\delta$ be small and $t<1$. Then for every $x\in B(0,\delta)$, the function is bigger than $t$ at $x$. Since 
\begin{equation*}
\int_{B(0,\delta)} |x|^\beta dx\sim \delta^{\beta+n},
\end{equation*}
we want $\delta^{\beta+n}\le K \delta^{\lambda}$ for small $\delta$. Therefore, $\beta+n\ge \lambda$ is necessary. 

Let $f$ be the characteristic function of the ball centered at $1$ with radius $1/2$. This function is clearly in $\mathcal L^{p,\lambda}(|x|^\beta)$ for any $\beta$ and for $0\le \lambda<n$. The maximal operator acting on $f$ satisfies $Mf(x)\ge c$ for some $c>0$ and $x\in B(0,1)$. Consequently, $Mf\notin W \mathcal L^{p,\lambda}(|x|^\beta)$ for  $\beta< \lambda-n$.   
\end{proof}

A direct proof of the necessity of $\beta<\lambda+n(p-1)$ is obtained as follows. 
For $\beta\ge \lambda+n(p-1)$, the function $|x|^{-n}\chi_{B(0,1)}$ is in $\mathcal L^{p,\lambda}(|x|^\beta)$ and is not locally integrable. Therefore,  $\beta<\lambda+n(p-1)$ is necessary.

\begin{proposition}\label{necrangehilbert}
Let $1\le p<\infty$ and $0\le \lambda <1$. If the Hilbert transform $H$ is bounded from $\mathcal L^{p,\lambda}(w)$ to $W\mathcal L^{p,\lambda}(w)$, then $w\in A_{p+\lambda}$.
\end{proposition}

\begin{proof}
First we observe that for characteristic functions of sets the norms in $\mathcal L^{p,\lambda}(w)$ and $W\mathcal L^{p,\lambda}(w)$ are the same.

Given an interval $I$,  let $I'$ be the adjacent interval of the same length, placed at the right of $I$. Note that for $x\in I'$, $|H(\chi_{I})(x)|>1/(2\pi)$. Assuming the weak boundedness of $H$ we have
\begin{equation*}
\aligned
\|\chi_{I'}\|_{W\mathcal L^{p,\lambda}(w)}&\le \|2\pi H(\chi_{I})\chi_{I'}\|_{W\mathcal L^{p,\lambda}(w)}\\ 
&\le 2\pi \|H(\chi_{I})\|_{W\mathcal L^{p,\lambda}(w)}\le C\|\chi_{I}\|_{\mathcal L^{p,\lambda}(w)}.
\endaligned
\end{equation*}
Interchanging the role of $I$ and $I'$ we deduce that $\|\chi_{I}\|_{W\mathcal L^{p,\lambda}(w)}$ and \break $\|\chi_{I'}\|_{W\mathcal L^{p,\lambda}(w)}$ are comparable.

For $\sigma$ to be chosen later, we observe that $|H(\sigma \chi_I)(x)|> {\sigma(I)}{(2\pi |I|)^{-1}}$ for $x\in I'$. Also 
\begin{equation*}
\frac {w(I)}{|I|^{\lambda}}\le \|\chi_{I}\|_{W\mathcal L^{p,\lambda}(w)}^p\le C \|\chi_{I'}\|_{W\mathcal L^{p,\lambda}(w)}^p,
\end{equation*}
where the first inequality holds by the definition of the norm. Then
\begin{equation*}
\aligned
\frac{\sigma (I)}{|I|}  \frac{w(I)^{1/p}}{|I|^{\lambda/p}} & \le C \left\|\frac{\sigma (I)}{|I|} \chi_{I'}\right\|_{W\mathcal L^{p,\lambda}(w)}\le C \|2\pi H(\sigma \chi_I)\chi_{I'}\|_{W\mathcal L^{p,\lambda}(w)}\\ 
&\le 2\pi C \|H(\sigma\chi_{I})\|_{W\mathcal L^{p,\lambda}(w)}\le C'\|\sigma\chi_{I}\|_{\mathcal L^{p,\lambda}(w)}.
\endaligned
\end{equation*}
Using Lemma \ref{emb5} we get
\begin{equation*}
\frac{\sigma (I)}{|I|}  \frac{w(I)^{1/p}}{|I|^{\lambda/p}}\le C'' \left(\int_I \sigma^{\frac p{1-\lambda}} w^{\frac 1{1-\lambda}}\right)^{\frac {1-\lambda}p},
\end{equation*}
from which $w\in A_{p+\lambda}$ follows if we choose $\sigma=\sigma^{\frac p{1-\lambda}} w^{\frac 1{1-\lambda}}$.  
\end{proof}

\begin{remark}
In the case of the Hilbert transform one could prefer to assume that it is defined a priori only for Schwartz functions through the principal value formula. The proof given here can be adapted to such assumption by approximating the involved functions by smooth ones. Without affecting the proof one can take the intervals $I$ and $I'$ separated by a distance equal to their length, instead of taking them adjacent, so that there is some room for the approximation.   
\end{remark}

The proof of Proposition \ref{necrangehilbert} can be adapted to higher dimensions to obtain the necessity of the condition $A_{p+\lambda/n}$ for the Riesz transforms and other singular integral operators satisfying a nondegeneracy condition (see \cite[Chapter V, \S 4.6]{Ste93} for the similar result in the Lebesgue setting).

\begin{proposition}\label{necrangepowerh}
Let $1\le p<\infty$ and $0\le \lambda <1$. If $H$ is bounded from $\mathcal L^{p,\lambda}(|x|^\beta)$ to $W\mathcal L^{p,\lambda}(|x|^\beta)$, then $\lambda-1< \beta<\lambda+p-1$. 
\end{proposition}

\begin{proof}
 The condition $\beta<\lambda+p-1$ is a consequence of $|x|^\beta\in A_{p+\lambda}$ as required by the previous proposition. 
 
As in the proof for the maximal operator in Proposition \ref{necrangepower}, the estimate $H(\chi_{1,2})(x)\ge 1/(2\pi)$ for $x\in (0,1)$ is enough to get $\beta\ge \lambda-1$. To rule out the Morrey estimate for $\beta\ge \lambda-1$ we consider $\chi_{(0,1)}$, which is in $\mathcal L^{p,\lambda}(|x|^{\lambda-1})$. A direct computation shows that  $H\chi_{(0,1)}(x)=\pi^{-1} \log (|x|/|x-1|)$. Then $|H\chi_{(0,1)}(x)|\ge c (-\log |x|) \chi_{(0,1/4)}(x)$, and the last function is not in $W\mathcal L^{p,\lambda}(|x|^{\lambda-1})$.  
\end{proof}

The result in this proposition was proved for the strong estimates (hence, $1<p<\infty$) by N.\ Samko in \cite[Theorem 4.7]{Sam09}, and in \cite{Sam13} she discussed the necessity of a more general condition. Our result gives also the necessity for the weak estimates ($1\le p<\infty$).

In the next proposition we give a direct proof of the boundedness of $M$ for the Morrey spaces with weight $|x|^{\lambda-n}$ (the endpoint of the allowed range). In \cite{Ta15} this is a consequence of a certain sufficient condition involving duality. 

\begin{proposition}\label{endp}
Let $0\le \lambda <n$. $M$ is bounded on $\mathcal L^{p,\lambda}(|x|^{\lambda-n})$ for $1<p<\infty$ and from $\mathcal L^{1,\lambda}(|x|^{\lambda-n})$ to $W\mathcal L^{1,\lambda}(|x|^{\lambda-n})$.
\end{proposition}

\begin{proof}
 Let $f\in \mathcal L^{p,\lambda}(|x|^{\lambda-n})$ with $\lambda>0$. Assume that it is nonnegative. Consider the ball $B_r:=B(x,r)$. Decompose $f$ as $f_1+f_2$, where $f_1=f\chi_{B_{2r}}$. Using the subadditivity of $M$ we have 
 \begin{equation*}
Mf(y)\le Mf_1(y)+Mf_2(y).
\end{equation*}
Using the boundedness of $M$ on $L^{p}(|x|^{\lambda-n})$ we have
\begin{equation*}
\int_{B_r}(Mf_1)(y)^p |y|^{\lambda-n}dy\le C_1 \int_{B_{2r}}f(y)^p |y|^{\lambda-n}dy\le C_2 r^{\lambda} \|f\|_{\mathcal L^{p,\lambda}(|x|^{\lambda-n})}^p.
\end{equation*}
On the other hand, $Mf_2(y)$ is almost constant on $B_r$ in the sense that $Mf_2(y_1)\break \sim Mf_2(y_2)\sim Mf_2(x)$ for $y_1,y_2\in B_r$. Then
\begin{equation}\label{cinco}
\int_{B_r}(Mf_2)(y)^p |y|^{\lambda-n}dy\le C_1 (Mf_2)(x)^p\int_{B_{r}}|y|^{\lambda-n}dy.
\end{equation}
Moreover,
\begin{equation*}
Mf_2(x)\sim\sup_{R>2r}\frac 1{R^n}\int_{B_R\setminus B_{2r}}f.
\end{equation*}
We distinguish two types of balls as in Remark \ref{select}. 

In the case of a ball centered at $0$, we have
\begin{equation*}
\frac 1{R^n}\int_{B_R\setminus B_{2r}}f\le \frac 1{R^n} \left(\int_{B_R}f^p\right)^{1/p} |B_R|^{1/p'}\le C\|f\|_{\mathcal L^{p,\lambda}(|x|^{\lambda-n})},
\end{equation*}
using \eqref{basicineq} with $w\equiv 1$ and $\alpha=\lambda-n$. Since the last integral in \eqref{cinco} is $Cr^\lambda$ we get the desired estimate. 

For a ball centered at $x\ne 0$ with radius $r\le |x|/4$ we consider first $R\ge |x|/2$. In such case, 
\begin{equation*}
\frac 1{R^n}\int_{B_R\setminus B_{2r}}f\le \frac 1{R^n}\int_{B(0,3R)}f\le C\|f\|_{\mathcal L^{p,\lambda}(|x|^{\lambda-n})}. 
\end{equation*}
The last integral in \eqref{cinco} is $C|x|^{\lambda-n}r^n$ and this is bounded by $Cr^{\lambda}$ because $r<|x|$ and $\lambda-n$ is negative. Let now $2r<R<|x|/2$. Then
\begin{equation*}
\frac 1{R^n}\int_{B_R\setminus B_{2r}}f\le \frac 1{R^n}\left(\int_{B_R}f^p\right)^{1/p}R^{n/p'} \le C|x|^{\frac{n-\lambda}p}R^{\frac{\lambda-n}p}\|f\|_{\mathcal L^{p,\lambda}(|x|^{\lambda-n})}. 
\end{equation*}
Replacing the last integral in \eqref{cinco} by $C|x|^{\lambda-n}r^n$ the needed estimate holds because $R^{\lambda-n}r^n \le r^\lambda$ due to $R>r$.

The proof of the weak type for $p=1$ is similar.  
\end{proof}

The range $\lambda-n< \beta<\lambda+n(p-1)$ for weights of type $|x|^\beta$ corresponds to all the power weights in $A_{p+\lambda/n}\cap RH_{n/(n-\lambda)}$. On the other hand, the result in Theorem \ref{teo1} is valid for $\lambda<n/\sigma'$, that is, we need $w\in RH_\sigma$ for some $\sigma>n/(n-\lambda)$ to get the estimate. Such a $\sigma$ exists for any weight in $RH_{n/(n-\lambda)}$ by the self-improvement property of the reverse H\"older inequalities (Gehring's lemma). The endpoint weight $|x|^{\lambda-n}$, for which the estimates for $M$ hold, is not in $RH_{n/(n-\lambda)}$, but it is in $RH_\sigma$ for every $\sigma<n/(n-\lambda)$. One could guess the necessity of a reverse H\"older condition of this type in Proposition \ref{necrange}, but we have not been able to get it. 

The weights in $A_{p+\lambda/n}\cap RH_{n/(n-\lambda)}$ are characterized by the factorization $u^{-\lambda/n} w$ with $u\in A_1$ and $w\in A_{p}\cap RH_{n/(n-\lambda)}$. The result in Theorem \ref{teo1} covers all the weights of this type for which $u$ is a power weight in $A_1$. By translation invariance power weights can be taken to be centered at a point different from the origin. The sufficiency of  $A_{p+\lambda/n}\cap RH_{n/(n-\lambda)}$ in  Theorem \ref{teo1} remains an open question for us.

\end{document}